\documentclass{amsart}
\usepackage{latexsym,amsmath}
\usepackage{amsthm,amssymb,mathrsfs}

\newcommand{\Aut}{\mathrm{Aut}}
\newcommand{\degree}{\mathrm{degree}}

\newcommand{\RR}{\mathbb{R}} 
\newcommand{\CC}{\mathbb{C}} 

\newcommand{\beq}{\begin{eqnarray*}}
\newcommand{\eeq}{\end{eqnarray*}}
\newcommand{\beg}{\begin{equation}}
\newcommand{\eeg}{\end{equation}}

\newtheorem{thm}{Theorem}[section]

\newtheorem{lem}[thm]{Lemma}
\newtheorem{prop}[thm]{Proposition}
\theoremstyle{definition}
\newtheorem{definition}[thm]{Definition}
\theoremstyle{remark}
\newtheorem{remark}[thm]{Remark}
\newtheorem{example}[thm]{Example}

\numberwithin{equation}{section}

\title{New examples of proper holomorphic maps among symmetric domains}
\author{Aeryeong Seo}

\address{School of Mathematics,
Korea Institute for Advanced Study (KIAS)
85 Hoegiro, Dongdaemun-gu, Seoul 130-722, Korea}
\email{Aileen83@kias.re.kr}


\begin{document}
\maketitle

\markboth{Aeryeong Seo}{New examples of proper holomorphic maps among symmetric domains}

\section{Introduction}

Let $\Omega_{r,s}$ be a bounded symmetric domain of type I which is defined by
$$
\Omega_{r,s} = \{ Z\in M(r,s,\mathbb{C}) : I_{r,r} -  Z Z^* >0\}.
$$
Here we denote by $>0$ positive definiteness of square matrices, by $M(r,s,\mathbb{C})$ the set of $r\times s$ complex matrices and by $I_{r,r}$ the $r\times r$ identity matrix.
Let  $D_{r,s}$ be a generalized ball which is defined by
$$
D_{r,s} = \{ [z_1, \dots, z_{r+s}]\in \mathbb{P}^{r+s-1} : |z_1|^2 + \dots +|z_r|^2 > |z_{r+1}|^2 + \dots + |z_{r+s}|^2 \}.
$$

\begin{definition}
\begin{enumerate}
\item Let $f,\,g : \Omega_1 \rightarrow \Omega_2$ be holomorphic maps between domains $\Omega_1,\, \Omega_2$.
We say $f$ and $g$ are equivalent if and only if $f = A\, \circ\, g \,\circ B$ for some $B\in \Aut(\Omega_1)$ and $A\in \Aut(\Omega_2)$.
\item Let $g_1,\, g_2 : \mathbb{P}^{n} \rightarrow \mathbb{P}^{N}$ be rational maps. We say $g_1$ and $g_2$ are rationally equivalent if there is a rational map $g: \mathbb{P}^{n} \rightarrow \mathbb{P}^{N}$ such that $g$ is a common extension of $g_1$ and $g_2$.
\end{enumerate}
\end{definition}
The aim of this paper is presenting a simple way to generate proper monomial rational
maps between generalized balls and via the relations between generalized balls and
bounded symmetric domains of type I given in \cite{Ng13}, giving new examples of proper holomorphic
maps between bounded symmetric domains of type I.

Consider a proper rational map $g: D_{r,s}\rightarrow D_{r',s'}$.
In homogeneous coordinate, put $g([z_1, \cdots,z_{r+s}] )  = [g_1,\cdots,g_{r'+s'}]$. Suppose that $g_i$ are  monomials in $z_1,\dots, z_{r+s}$ for each $i, \,1\leq i\leq r'+s'$. Then we can define the homogeneous polynomial $P : \RR^{r+s} \rightarrow \RR$ satisfying
\begin{equation}\label{P}
P(|z_1|^2,\dots,|z_{r+s}|^2) = \sum_{k=1}^{r'} |g_k|^2-\sum_{k=r'+1}^{r'+s'} |g_k|^2.
\end{equation}
Since $g$ is proper, $P(x)=0$ whenever $\sum_{j=1}^r x_j=\sum^{r+s}_{j=r+1} x_j$. Hence $P$ should be of the form
\begin{eqnarray} \label{form of P}
\left(\sum_{j=1}^r x_j - \sum^{r+s}_{j=r+1} x_j \right)^m Q_P(x)
\end{eqnarray}
for some positive integer $m$ and homogeneous polynomial $Q_P(x)$.
\begin{thm}\label{classify ball}
Let $g: D_{r,r} \rightarrow D_{r+1,r+1},\, (r\geq 2)$ be a proper monomial rational map. Then $g$ is rationally equivalent to one of the following up to automorphisms of $D_{r,r}$ and $D_{r+1,r+1}:$
\begin{enumerate}
\item  In case of $\degree(g)=1 :$ \\ $g([z_1,\dots, z_{2r}]) = [z_1,\dots, z_{r},\phi(z), z_{r+1},\dots, z_{2r},\phi(z)]$ where $\phi(z)$ is a degree one homogeneous polynomial in $z_1,\dots, z_{2r}$ \label{standard}
\item In case of $\degree(g)=2 :$
\begin{enumerate}
\item $g([z_1,z_2,z_3,z_4]) = [z_1^2, z_1z_2, z_2z_3, z_3^2 ,z_1z_4,  z_3z_4]$ \label{me}
\item $g([z_1,z_2,z_3,z_4]) = [z_1^2, \sqrt{2} z_1z_2, z_2^2, z_3^2 ,\sqrt{2} z_3z_4,  z_4^2]$ \label{ng}
\end{enumerate}
\item In case of $\degree(g) \geq 3 :$ if  $Q_P(x)$ has degree 1 or,
the coefficients of the polynomial $Q_P(x)$ are nonnegative, there is no proper monomial rational map.
\end{enumerate}
\end{thm}

The condition in Theorem \ref{classify ball} about $Q_P$ are due to combinatorial method counting monomials in expansion of multiplied polynomial.

The method to characterize proper monomial rational maps originally comes from
J. P. D'Angelo in [1]. He studied proper monomial holomorphic maps from
the unit ball to the higher dimensional unit ball via characterizing the polynomials which
can be obtained by taking Euclidean norm on proper maps. By characterizing these
polynomials, he obtained complete list of proper monomial holomorphic maps from the two
dimensional unit ball to the four dimensional unit ball. In this paper, we modify this polynomial
which is appropriate to proper monomial rational maps between generalized balls and characterize
the polynomial by counting the number of monomials in the polynomial.

For bounded symmetric domains of rank at least 2, properties of proper holomorphic maps are deeply related to special kind of totally geodesic subspaces of given domains which are called \textit{invariantly geodesic subspaces}. These are  totally geodesic submanifolds with respect to the Bergman metric which are still totally geodesic under the action of automorphisms of the compact dual of ambient domain.
Invariantly geodesic subspaces first appeared in \cite{MokTsai} as far as the author knows. These subspaces  play important roles to characterize proper holomorphic maps between bounded symmetric domains. Especially N. Mok and I. H. Tsai proved that proper holomorphic maps between irreducible bounded symmetric domains preserve the maximal characteristic subspaces which are also invariantly geodesic subspaces. Based on \cite{MokTsai, Tsai}, the rigidity of irreducible bounded symmetric domains have been developed and incorporated by Z. Tu \cite{Tu02_1,Tu02_2} and S. C. Ng \cite{Ng13, Ng15}. Especially, Ng \cite{Ng13} found that generalized balls in the projective spaces parametrize the maximal invariantly geodesic subspaces of bounded symmetric domains of type I and we use this relation to find several examples of proper holomorphic maps between bounded symmetric domains of type I.

Consider the subspaces in $\Omega_{r,s}$ of the form $$L_{[A,B]} = \{Z\in \Omega_{r,s} : AZ = B\}$$ where $A\in M(1,r,\mathbb{C})$, $b\in M(1,s,\mathbb{C})$ satisfying $[A,B]\in D_{r,s}$ which are totally geodesic under the action of $SL(r+s,\CC)$.
These are the maximal invariantly geodesic subspaces.
For $X = [A,B] \in D_{r,s}$, denote $X^\# = L_X$.

For a proper holomorphic map $ f: \Omega_{r,r}\rightarrow \Omega_{r+1,r+1},\,(r\geq 2)$
which preserves the maximal invariantly geodesic subspaces,
there is a proper holomorphic map $g : D_{r,r}\rightarrow D_{r+1,r+1}$ such that $f(X^\#) \subset g(X)^\#$
 for generic $X\in D_{r,r}$.

\begin{thm}\label{main theorem}
Let $ f: \Omega_{r,r}\rightarrow \Omega_{r+1,r+1},\,(r\geq 2)$ be a proper holomorphic map.
Suppose that $f$ preserves the maximal invariantly geodesic subspaces
and an induced proper holomorphic map $g : D_{r,r}\rightarrow D_{r+1,r+1}$ satisfies the conditions in Theorem \ref{classify ball}.
Then
 $f$ is equivalent to one of the following$\colon$
\begin{enumerate}
\item $f(Z) = \left(
           \begin{array}{cc}
             Z & 0 \\
             0 & h(Z) \\
           \end{array}
         \right)$ for $Z\in \Omega_{r,r}$ \\and for some holomorphic map $h: \Omega_{r,r}\rightarrow \Delta = \{z\in \CC : |z| <1\}$. \label{standard map}
\item  $f\left(\left(
         \begin{array}{cc}
           z_1 & z_2 \\
           z_3 & z_4 \\
         \end{array}
       \right)
    \right) = \left(
           \begin{array}{ccc}
             z_1^2 & z_1z_2& z_2  \\
             z_1z_3 & z_2z_3& z_4  \\
             z_3 & z_4& 0  \\
           \end{array}
         \right)$,\,\,\,
for $\left(
         \begin{array}{cc}
           z_1 & z_2 \\
           z_3 & z_4 \\
         \end{array}
       \right) \in \Omega_{2,2}$ \label{map me}
\item $f\left(\left(
         \begin{array}{cc}
           z_1 & z_2 \\
           z_3 & z_4 \\
         \end{array}
       \right)
    \right) = \left(
           \begin{array}{ccc}
             z_1^2 & \sqrt{2}z_1z_2 & z_2^2 \\
             \sqrt{2} z_1z_3 & z_1z_4+z_2z_3 & \sqrt{2}z_2z_4 \\
             z_3^2 & \sqrt{2}z_3z_4 & z_4^2 \\
           \end{array}
         \right)$ \label{map ng}
\end{enumerate}
 \end{thm}

Here is the outline of the paper.
Section 2 introduces some basic terminology, well-known facts and the invariantly geodesic subspaces.
In section 3, we modify D'Angelo's method to proper monomial maps between generalized balls
and classify the maps which are needed to sort proper holomorphic maps between bounded symmetric domains of type I. We count the number of monomials in homogeneous polynomial which is multiplied by two homogeneous polynomials.
In Section 4, we present a way to generate proper holomorphic maps from $\Omega_{r,s}$ to $\Omega_{r',s'}$ and prove Theorem \ref{main theorem}. Furthermore we give more examples which are interesting.

\subsection*{Acknowledgement}{The author thanks to professor Ngaiming Mok for introducing the problem. The author would like to thank professors Sui-Chung Ng and Sung-yeon Kim for invaluable advice and discussion on this work. Especially Ng first found the map \eqref{map ng} in Theorem \ref{main theorem}. This research was supported by National Researcher Program of the National Research Foundation (NRF)
funded by the Ministry of Science, ICT and Future Planning(No.2014028806).
}

\section{Preliminary}
\subsection{Basic facts and Terminology}
 At first, we introduce terminology and some facts. For more detail, see \cite{Ng13, MokTsai}.
Let $G_{r,s}$ be the Grassmannian of $r$-planes in $r+s$ dimensional complex vector space $\mathbb{C}^{r+s}$ which is the compact dual of $\Omega_{r,s}$. For $X\in M(r,r+s,\mathbb{C})$ of rank $r$, denote $[X]$ an $r$-plane in $\mathbb{C}^{r+s}$ which is generated by row vectors of $X$. For each element $Z$ in $\Omega_{r,s}$, there corresponds an $r$-plane $[I_{r,r},Z] \in G_{r,s}$. This is the Borel embedding of $\Omega_{r,s}$ into $G_{r,s}$. It is clear that $SL(r+s,\mathbb{C})$ acts holomorphically and transitively $G_{r,s}$. Denote $SU(r,s)$ the subgroup of $SL(r+s,\CC)$ satisfying $ M \left(
                                 \begin{array}{cc}
                                   -I_{r,r} & 0 \\
                                   0 & I_{s,s} \\
                                 \end{array}
                               \right) M^* = \left(
                                 \begin{array}{cc}
                                   -I_{r,r} & 0 \\
                                   0 & I_{s,s} \\
                                 \end{array}
                               \right)$ for all $M\in SU(r,s)$.
Then $SU(r,s)$ is the automorphism group of $\Omega_{r,s}$. If we put $M = \left(
                                 \begin{array}{cc}
                                   A & B \\
                                   C & D \\
                                 \end{array}
                               \right)$ where $A\in M(r,r,\CC)$, $B \in M(r,s,\CC)$, $C\in M(s,r,\CC)$,
                               $D\in M(s,s,\CC)$,
$M$ acts on $\Omega_{r,s}$ by $Z\mapsto (A + ZC)^{-1}(B + ZD)$. From now on, if we write
$M = \left(
                                 \begin{array}{cc}
                                   A & B \\
                                   C & D \\
                                 \end{array}
                               \right)\in SU(r,s)$, without ambiguity, $A,\,B,\,C,\,D$ are block matrices of the above form.
\subsection{Invariantly geodesic subspaces in $\Omega_{r,s}$}

Consider a complex submanifold $S$ in $\Omega_{r,s}$. For every $g\in SL(r+s,\mathbb{C})$ such that $g(S)\cap \Omega_{r,s}\neq \emptyset$, if the submanifold $g(S)\cap \Omega_{r,s}$ is totally geodesic in $\Omega_{r,s}$ with respect to the Bergman metric of $\Omega_{r,s}$,  then $S$ is called \textit{invariantly geodesic subspace} of $\Omega_{r,s}$. In particular, for $W\in \Omega_{r',s'}$ with $r'\leq r$ and $s'\leq s$, the image of the embedding
$\mathrm{i} : W \mapsto  \left(
                      \begin{array}{cc}
                        0 & 0 \\
                        0 & W \\
                      \end{array}
                    \right) \in \Omega_{r,s}$
is an invariantly geodesic subspace of $\Omega_{r,s}$. The totally geodesic subspaces which are equivalent under the action of $SU(r,s)$ to $\mathrm{i}(\Omega_{r,s})$ in $\Omega_{r,s}$ are called $(r',s')$\textit{-subspaces} of $\Omega_{r,s}$. Among these $(r',s')$-subspaces, the maximal invariantly geodesic subspaces are parametrized by the generalized ball in $\mathbb{P}^{r+s-1}$.
\begin{prop} [\cite{Ng13}]
The subspaces of the form
\begin{eqnarray}\label{invariantly geodesic}
L = \{Z\in \Omega_{r,s} : AZ = B\}
\end{eqnarray}
where $A\in M(1,r,\mathbb{C})$, $B\in M(1,s,\mathbb{C})$ satisfying $[A,B]\in D_{r,s}$ are $(r-1,s)$-subspaces.
\end{prop}
For example, in case of invariantly geodesic subspaces
$$\left\{ \left(
                \begin{array}{c}
                        0  \\
                        W \\
                      \end{array}
                    \right) \in \Omega_{r,s} : W\in \Omega_{r-1,s}\right\},$$
$A = (1,0,\dots,0)\in M(1,r,\mathbb{C})$ and $B = (0,\dots , 0)\in M(1,s,\mathbb{C})$.

For $\Omega_{r,s}$ and $D_{r,s}$, consider the two surjective maps
\begin{eqnarray} \label{fiber images}
&&\phi: \mathbb{P}^{r-1} \times \Omega_{r,s} \rightarrow \Omega_{r,s}, \,\, ([X],Z)\mapsto Z\\
&&\psi: \mathbb{P}^{r-1} \times \Omega_{r,s} \rightarrow D_{r,s}, \,\, ([X],Z)\mapsto [X,XZ].
\end{eqnarray}
For $Z\in \Omega_{r,s}$, denote $Z^\# = \psi (\phi^{-1}(Z)) \subset D_{r,s}$. Similarly for $X\in D_{r,s}$, denote $X^\# = \phi ( \psi^{-1}(X))\subset \Omega_{r,s}$. $Z^\#$ and $X^\#$ are called fibral images of $Z$ and $X$ respectively. Then for $ Z\in \Omega_{r,s}$ and $ X = [A,B] \in D_{r,s}$ where $A\in M(1,r,\mathbb{C})$ and $B\in M(1,s,\mathbb{C})$,
\begin{eqnarray}
Z^\# = \{[A,AZ]\in D_{r,s} : [A]\in \mathbb{P}^{r-1}\} \cong \mathbb{P}^{r-1}\\
X^\# = \{Z\in \Omega_{r,s} : AZ = B\} \cong (r-1,s)\text{-subspace}
\end{eqnarray}

\begin{prop}[cf. \cite{Ng13}] \label{proper moduli}
Let $f : \Omega_{r,r} \rightarrow \Omega_{r+1,r+1}$ be a proper holomorphic map. Suppose that there is a meromorphic map $g : D_{r,r} \rightarrow D_{r+1,r+1}$ such that $f(X^\#)\subset g(X)^\#$ for generic point $X\in D_{r,r}$. Then $g$ is a proper map or
$f(Z)=\left(
  \begin{array}{cc}
    Z & 0 \\
    0 & h(Z) \\
  \end{array}
\right)$ for some holomorphic function $h: \Omega_{r,r}\rightarrow \Delta$.
\end{prop}

\section{Proper monomial rational map from $D_{r,s}$ to $D_{r',s'}$}

Let $g: D_{r,s}\rightarrow D_{r',s'}$ be a proper monomial rational map and $P,\, Q_P$ be homogeneous polynomials
defined by \eqref{P} and \eqref{form of P}.
Then $Q_P$ has the following properties:
\begin{enumerate}
\item $Q_P(x)$ is a homogeneous polynomial  which is not identically zero on
$$\left\{x = \{x_1,\dots,x_{r+s})\in \mathbb{R}^{r+s} : \sum_{j=1}^r x_j=\sum^{r+s}_{j=r+1} x_j\right\}$$
\item  $Q_P(x)>0$ whenever $x_i >0$ for all $i$ and $\sum_{j=1}^{r} x_j> \sum_{j=r+1}^{r+s}x_j$.
\end{enumerate}

\subsection{Classifying proper monomial rational map from $D_{r,r}$ to $D_{r+1,r+1}$}
A situation of classifying proper rational maps between generalized balls is different from that of classifying proper holomorphic maps between unit balls in \cite{D'Angelo88} since there are infinite number of proper rational maps which are same in open dense subset. For example, $g : D_{2,2}\rightarrow D_{3,3}$, $[z_1,\dots, z_4]\mapsto [z_1 h, z_2h , 0, z_3h, z_4h, 0]$ for any holomorphic function $h$ of $\mathbb{C}^4$ which is not identically zero on $D_{2,2}$ are same in open dense subset depending on zero set of $h$. On the other hand, proper rational maps which are same in open dense subset induce the same proper holomorphic map between corresponding bounded symmetric domains of type I.  Hence we consider equivalence relation on proper monomial rational maps to  incorporate these infinite number of rational maps.
\begin{definition}
Let $g_1,\, g_2 : \mathbb{P}^{2r-1} \rightarrow \mathbb{P}^{2r+1}$ be rational maps. We say $g_1$ and $g_2$ are rationally equivalent if there is a rational map $g: \mathbb{P}^{2r-1} \rightarrow \mathbb{P}^{2r+1}$ such that $g$ is a common extension of $g_1$ and $g_2$.
\end{definition}

We may assume that every components of $g : D_{r,s}\rightarrow D_{r',s'}$ have no common factor.

In the rest of this section, we characterize the induced polynomial $P(x)$ and the proper monomial rational maps from $D_{r,r}$ to $D_{r+1,r+1}$ to prove Theorem \ref{classify ball}.
For this aim, we will count the number of monomials of $P$ for suitable $Q_P$.
For a polynomial $A$, denote $n_i(A)$ the number of monomials with maximal degree in  $x_i$ of $A$ and $n(A)$ the number of monomials in $A$.
\begin{lem}
For polynomial $A = (b_1x_1 + \dots + b_kx_k)^m \widetilde{A}$ with nonzero polynomial $\widetilde{A}$, positive integer $m$ and nonzero $b_i$ for all $i,\, 1\leq i\leq k$, $n(A) \geq \sum_{i=1}^k n_i(\widetilde{A})$.
\begin{proof}
$(b_ix_i)^m$ times the monomial with the maximal degree of $x_i$ in $\widetilde{A}$ cannot be canceled.
\end{proof}
\end{lem}
\begin{lem} \label{m>1}
Let $P(x)$ be a homogeneous polynomial on $\mathbb{R}^{k}$ of the form $$(b_1x_1+\dots + b_kx_k)^mQ_P(x)$$ for some positive integer $m$, nonzero $b_i$ for all $i,\,1\leq i \leq k$  and homogeneous polynomial  $Q_P(x)$ with nonnegative coefficients. Then if $m \geq 2$, $n(P) \geq 2k-1$.
\end{lem}
\begin{proof}
Without loss of generality, we may assume that $Q_P(x)$ contain $x_1$ variable with $b_1>0$ and $n(Q_P)\geq2$.
Let $Q_P(x) = A_0 + A_1x_1 + A_2x_1^2 + \cdots + A_\alpha x_1^\alpha$ be the expansion of $Q_P(x)$ with respect to the degree of $x_1$ variable where $\alpha$ is the maximal degree of $x_1$ in $Q_P(x)$, $A_l$ is a homogeneous polynomial without $x_1$ variable having nonnegative coefficients and $A_0$ and $A_\alpha$ are nonzero. Denote $B= b_2 x_2+\dots+ b_k x_k$. Then
\begin{eqnarray}
P(x)=A_0B^m + x_1B^{m-1} ( mb_1A_0 + A_1B) + \dots + x_1^{\alpha +m}A_\alpha.\nonumber
\end{eqnarray}
Note that there are at least $k-1$ monomials in $A_0B^m$ and 1 monomial in $x_1^{\alpha +m}A_\alpha$.
Notice that the second term $x_1B^{m-1} ( mb_1A_0 + A_1B)$ is not vanish and has at least $k-1$ monomials.
Hence summing up, there are at least $2k-1$ monomials in $P$ when $m\geq 2$.
\end{proof}

\begin{lem}\label{nonnegative}
Let $P(x)$ be a homogeneous polynomial on $\mathbb{R}^{2r}$ of the form $$(x_1 +\dots+x_r -x_{r+1} - \dots -x_{2r})\,Q_P(x)$$ for some homogeneous polynomial  $Q_P(x)$ with nonnegative coefficients and $n(Q_P) \\\geq 2$. Then
\begin{enumerate}
\item $n(P) \geq 3r-1$  if $\,r\geq 2$,
\item $n(P) \geq 9\,$    if $\,r=3$.
\end{enumerate}
\end{lem}
\begin{proof}
As in the proof of Lemma \ref{m>1}, consider
$$P(x) = A_0B + x_1(A_0 + A_1B) + x_1^2 (A_1 + A_2B) + \dots + A_\alpha x_1^{\alpha+1}.$$
Suppose $A_i=0$ but $A_{i+1} \neq 0$ for some $i, 1\leq i\leq \alpha-1$. Then the coefficient of $x_1^{i+1}$ is $A_{i+1}B$ and then there exist at least $2r-1$ monomials which cannot be canceled. This implies that in this case, $n(P) \geq 4r-1$. Hence it is enough to consider when $A_i \neq 0$ for any $i,\, 0\leq i \leq \alpha$. In this case, there are at least $2r-1$ monomials in $A_0B$, $r-1$ monomials in $x_1(A_0 + A_1B)$, $r-1$ monomials in $x_1^2 (A_1 + A_2B)$ and 1 monomial in  $A_\alpha x_1^{\alpha+1}$. Hence $n(P) \geq 3r-1$.

 Consider $r=3$. We may assume that $A_i \neq 0$ for all $i$.  Since $n(A_i + A_{i+1})\geq 2$ for all $i$, it is enough to consider when $\alpha =1$. Then $P(x) = A_0B + x_1(A_0 +A_1B) + A_1x_1^2$. If $A_0 = A_1(x_4+x_5+x_6)$, then $n(A_0B) \geq 9$ and if $A_0 \neq A_1(x_4+x_5+x_6)$, then $n(x_1(A_0 +A_1B))\geq 3$. Hence $n(P) \geq 9$.
\end{proof}

\begin{lem}\label{degree r poly}
Let $P(x)$ be a nonzero homogeneous polynomial on $\mathbb{R}^{k}\,(k\geq 1)$ of the form $$(b_1x_1+\dots + b_kx_k)^m(a_1x_1 + \dots + a_kx_k)$$  for some positive integer $m$, $a_i \in \mathbb{R}$ for $i,\, 1\leq i \leq k$ and nonzero $b_i$ for all $i,\, 1\leq i \leq k$ . Then
\begin{enumerate}
\item if $m\geq 2$, then $n(P) \geq 2k-1$
\item If $m=1$ and $n(a_1x_1 + \dots + a_kx_k)\geq 2$, then $n(P) \geq 2k-2$.
\end{enumerate}
\end{lem}
\begin{proof}
We will prove (1). The proof of (2) is similar.

If $n(a_1x_1 + \dots + a_kx_k)=1$, then there are
$\left(
  \begin{array}{c}
     k+m-1\\
     m\\
  \end{array}
\right) \geq 2k-1$ number of monomials in $P$.

Suppose that  $n(a_1x_1 + \dots + a_kx_k)\geq 2$. We may assume that $a_1 \neq 0$. Put $A = a_2x_2 + \dots + a_kx_k$ and $B = b_2x_2+\dots + b_kx_k$. Then
$$
P(x) = B^mA+ x_1B^{m-1}(mb_1A + a_1B ) + \dots + a_1 x_1^{m+1}.
$$

Consider the case $mb_1A + a_1B \neq 0$. Let $x$ be the number of $a_i$'s which are zero and $y$ be the number of $a_i$'s which are nonzero.
Then $n(B^mA) \geq y-1+x(y-1) = -y^2 + (k+2)y-1-k$ for $y,\, 2\leq y \leq k$. At $y=2$ the minimum $k-1$ appears. Hence $n(P)\geq n(B^mA)+n(B^{m-1}(mb_1A + a_1B ))+ n( a_1 x_1^{m+1}) \geq 2k-1$.

If $mb_1A + a_1B = 0$, $n(B^mA) = n(B^{m+1}) =\left(
  \begin{array}{c}
     k+m\\
     m\\
  \end{array}
\right) \geq 2k-1$.
\end{proof}

\begin{lem}\label{degree1poly}
Let $P(x) = (x_1+x_2-x_3-x_4)\, Q_P(x)$ for $Q_P(x)=a_1x_1 + a_2x_2+a_3x_3+a_4x_4, \,\,a_i\in \mathbb{R},\,\, i=1,2,3,4$. Suppose that $n(P)\leq 6$ and
\begin{equation}\label{condition of Q}
Q_P(x)>0 \text{ whenever  } x_1+x_2>x_3+x_4 \text{ and } x_i>0 \text{ for all }i, 1\leq i\leq 4,
\end{equation}
then the $Q_P(x)$ is one of the following up to multiplication of constants:
$$
x_1,\,\, x_2,\,\, x_3,\,\, x_4,\,\, x_1+x_3,\,\,x_1+x_4,\,\, x_2+x_3,\,\, x_2+x_4,\,\, x_1+x_2+x_3+x_4$$
\end{lem}

\begin{proof}
We prove the lemma case by case.
\begin{eqnarray}
P(x) =&& a_1x_1^2 + a_2x_2^2 - a_3x_3^2 - a_4x_4^2 + (a_2 + a_1)x_1x_2 + (a_3 - a_1)x_1x_3 \nonumber\\
&& + (a_3-a_2)x_2x_3 + (a_4-a_1)x_1x_4 + (a_4-a_2)x_2x_4 - (a_3+a_4)x_3x_4
\end{eqnarray}
\begin{enumerate}
\item If only one $a_i$ is zero and the others are nonzero, then $Q_P$ is $x_i$ for $1\leq i \leq 4$.
\item If $a_1=0$ and $a_i\neq 0$ where $2\leq i \leq 4$, then there are monomials, $x_1x_i,\,\, x_i^2 $ for $2\leq i \leq 4$ which cannot be canceled. Hence $a_2=a_3,\,\, a_2 = a_4, \,\, a_4 + a_3=0$ and this is a contradiction. If $a_j = 0$ and $a_k \neq 0$ for $k\neq j$, by the same way, this cannot happen.
\item If $a_1=a_2=0,\,\,a_3\neq 0 ,\,\, a_4\neq 0$, then $a_3+a_4=0$. This contradicts to the condition \eqref{condition of Q}. Similarly, there is no $Q_P$ for $a_3=a_4=0,\,\,a_1\neq 0 ,\,\, a_2\neq 0$.
\item If $a_2=a_4=0,\,\,a_1\neq 0 ,\,\, a_3\neq 0$, then $a_1=a_3$ and $a_1>0$.
This case corresponds to $Q_P(x) = x_1+ x_3$ and similarly, cases, $\{a_1=a_3=0,\,\,a_2\neq 0 ,\,\, a_4\neq 0\}$, $\{a_1=a_4=0,\,\,a_3\neq 0 ,\,\, a_2\neq 0\}$, $\{a_3=a_2=0,\,\,a_1\neq 0 ,\,\, a_4\neq 0\}$ corresponds to $x_2+x_4,\,\,x_3+ x_2, \,\, x_1+x_4$ respectively.
\item If all $a_i$ are nonzero, by \eqref{condition of Q} $a_1>0$, $a_2 >0$. Hence at least 3 monomial among $(a_3 - a_1)x_1x_3,\, (a_3-a_2)x_2x_3,\, (a_4-a_1)x_1x_4,\, (a_4-a_2)x_2x_4 $ should be zero.
    This implies that $a_1=a_2=a_3=a_4$.
\end{enumerate}
\end{proof}

\begin{proof}[Proof of Theorem~\ref{classify ball}]
 Let $$(x_1 +\dots+x_r -x_{r+1} - \dots -x_{2r})^m\,Q_P(x)$$ be the homogeneous polynomial induced by $g$ for some positive integer $m$ and homogeneous polynomial  $Q_P(x)$. Then $P$ satisfies $n(P) \leq 2r +2$. If $n(Q_P)=1$, $g$ is rationally equivalent to \eqref{standard}. Hence we only need to consider when $n(Q_P)\geq 2$.

Suppose $m\geq 2$. Then by Lemma  \ref{m>1} and \ref{degree r poly},
$n(P)\geq 4r-1 > 2r+2$. Hence $m=1$. On the other hand, by Lemma \ref{degree r poly} and \ref{nonnegative}, $n(P)\geq 2r+2$ for all $r\geq 3$.

For $m=1,\, r=2$, by Lemma \ref{degree1poly},
\begin{eqnarray*}
&&x_1+x_2 -x_3 -x_4 ,\:
x_1^2 + x_1x_2 + x_2x_3 - x_3^2 -x_1x_4 - x_3x_4, \\
 &&x_2^2 + x_1x_2 + x_1x_4 - x_4^2 -x_2x_3 - x_3x_4,\:
x_1^2 + x_1x_2 + x_2x_4 - x_4^2 -x_1x_3 - x_3x_4, \\
&&x_2^2 + x_1x_2 + x_1x_3 - x_3^2 -x_2x_4 - x_3x_4,\:
x_1^2 + 2x_1x_2 + x_2^2 - x_3^2 -2x_3x_4 - x_4^2.
\end{eqnarray*}
Then the first one induces \eqref{standard} and the last one induce the map \eqref{ng}. The second to fifth one induce the map equivalent to \eqref{me}.
\end{proof}

\section{Proper holomorphic maps between bounded symmetric domains}

\subsection{Construction of proper holomorphic maps from $\Omega_{r,s}$ to $\Omega_{r',s'}$}
In this section, using the relations between $(r-1,s)$-subspaces in $\Omega_{r,s}$ and projective subspaces ($\cong \mathbb{P}^{r-1}$) in $D_{r,s}$ which is given in \cite{Ng13}, we describe the construction of proper holomorphic mapping between bounded symmetric spaces of type I.
To consider the boundary behavior of $g$, extend $\phi$ and $\psi$ to
\begin{eqnarray}
&&\tilde{\phi}: \mathbb{P}^{r-1} \times \overline\Omega_{r,s} \rightarrow \overline\Omega_{r,s}, \,\, ([X],Z)\mapsto Z\nonumber\\
&&\tilde{\psi}: \mathbb{P}^{r-1} \times \overline\Omega_{r,s} \rightarrow \overline{D}_{r,s}, \,\, ([X],Z)\mapsto [X,XZ].\nonumber
\end{eqnarray}
For the boundary points, consider the fibral image with respect to this extended map.
Let $z\in \partial\Omega_{r,s}$. This implies that $z$ satisfies $I_{r,r} - z\overline{z}^t \geq 0$ and there is $a\in \mathbb{C}^r$ such that $a(I_{r,r} - z\overline{z}^t )\overline{a}^t =0 $. Hence $z^\#$ may not be contained in $\partial D_{r,s}$ and
\begin{eqnarray}
z^\# \cap \partial D_{r,s} = \left\{ [a,az]\in \overline{D}_{r,s} : [a]\in \mathbb{P}^{r-1}, a(I_{r,r} - z\overline{z}^t )\overline{a}^t =0\right\}
\end{eqnarray}
On the other hand, for $[a,b] \in \partial D_{r,s}$ where $ a\in M(1,r,\mathbb{C})$ and  $ b\in M(1,s,\mathbb{C})$,
if $z\in [a,b]^\#$, $a\overline{a}^t = b\overline{b}^t = az\overline{(az)}^t = az \overline{z}^t\overline{a}^t$. Hence for $[a,b]\in \partial D_{r,s}$, $[a,b]^\# \subset \partial \Omega_{r,s}$.

\begin{definition}
For a rational map $g : D_{r,s}\rightarrow D_{r',s'}$, we say rational map $g$ is \textit{proper}
if for any point $x\in \partial D_{r,s}$ and open neighborhood $U$ of $x$
which does not intersect the indeterminacy of $g$, $g$ is proper on  $U\cap D_{r,s}$.
\end{definition}

\begin{prop} \label{proper}
Let $f : \Omega_{r,s} \rightarrow \Omega_{r',s'}$ be a holomorphic map. Suppose that there is a proper rational map $g : D_{r,s} \rightarrow D_{r',s'}$ satisfying
\begin{eqnarray} \label{fiber preserving}
f(X^\#)\subset g(X)^\# \,\,\text{ for generic point } X\in D_{r,s}.
\end{eqnarray}
Then $f$ is proper.
\end{prop}
\begin{proof}
Let $\{Z_j\}$ be a sequence in $\Omega_{r,s}$ such that $Z_j \rightarrow z\in \partial \Omega_{r,s}$. Choose points $X_j \in Z_j^\#$ and $x\in \partial D_{r,s}\cap z^\#$ such that $X_j \rightarrow x$.
Then since $g(X_j)\rightarrow g(x)$, $f(Z_j) \in f(X_j^\#) \subset g(X_j)^\# \rightarrow g(x)^\# \subset \partial \Omega_{r',s'}$. Hence $f$ is proper.
\end{proof}

Let $f : \Omega_{r,s} \rightarrow \Omega_{r',s'}$ be a proper holomorphic  maps  which is provided from a proper rational map $g : D_{r,s} \rightarrow D_{r',s'}$ satisfying the condition in Proposition \ref{proper}.
Denote $g = [g_1, g_2]$ where $g_1$ has $r'$-components and $g_2$ has $s'$-components.
For $X=[A,B]\in D_{r,s}$ and $Z\in X^\#$ i.e. $B=AZ$.
Then
$f([A,AZ]^\#) \subset g([A,AZ])^\#$ and this implies that
\begin{eqnarray} \label{system of equations}
g_1([A,AZ]) f(Z) = g_2([A,AZ]) \text{ for all } A\in \mathbb{P}^{r-1}.
\end{eqnarray}
\begin{prop} \label{indep vectors}
Let $g = [g_1,g_2] : D_{r,s} \rightarrow D_{r',s'}$ be a proper rational map. Let $f : M(r,s,\CC)\rightarrow M(r',s',\CC)$ be a holomorphic map satisfying \eqref{system of equations}.
Suppose that for generic points $Z\in \Omega_{r,s}$, there are $r'$ points $\{X_i: 1\leq i\leq r'\}$ in $\mathbb{P}^{r-1}$ such that
$\{g_1([X_i,X_i Z]): 1\leq i\leq r'\}$ are independent as $r'$ vectors in $\CC^{r'}$.
Then $f(\Omega_{r,s}) \subset \Omega_{r',s'}$.
\end{prop}
\begin{proof}
By \eqref{system of equations}, for $Z\in \Omega_{r,s}$
$$g_1([X_i,X_i Z])\left( I_{r',r'} - f(Z)f(Z)^* \right) g_1([X_i,X_i Z])^* >0.$$
Hence if $\{g_1([X_i,X_i Z]): 1\leq i\leq r'\}$ are independent, $I_{r',r'} - f(Z)f(Z)^*$ is positive definite. This implies the proposition.
\end{proof}

Hence for a proper rational map $g$ satisfying the condition in Proposition \ref{indep vectors}, if we find a solution of the system of equations \eqref{system of equations}, we get a proper holomorphic maps by Proposition  \ref{proper}.

\begin{remark}
For a meromorphic map $g : D_{r,s} \rightarrow D_{r',s'}$ and a holomorphic map $f : \Omega_{r,s}\rightarrow \Omega_{r',s'}$  satisfying \eqref{fiber preserving}, put $g'$ a meromorphic map $ h\circ g_2 \circ h'$ for some $h'\in \text{Aut}(D_{r,s})$ and $h \in\text{Aut}(D_{r',s'})$. Then there is $H\in \text{Aut}(\Omega_{r,s})$ and $H'\in \text{Aut}(\Omega_{r',s'})$ such that $g'$ and $f' := H'\circ f\circ H$ satisfies \eqref{fiber preserving}. This is due to the construction of \eqref{fiber images} and for more detail, see \cite{Ng13}.
\end{remark}

\subsection{Proof of Theorem \ref{main theorem}}
Note that two rationally equivalent proper monomial rational maps from $D_{r,r}$ to $D_{r+1,r+1}$ induce the same proper holomorphic map from $\Omega_{r,r}$ to $\Omega_{r+1,r+1}$.
By Theorem \ref{classify ball}, there are three possibilities to be $g$.
\eqref{ng} and \eqref{me} satisfies the condition in Proposition \ref{indep vectors}.
We will only induce the proper map \eqref{me} since calculation of map \eqref{ng} is similar.
Proper rational map is given by $g([x_1,\,x_2,\,x_3,\,x_4]) = [x_1^2,\, x_1x_2,\, x_2x_3,\, x_3^2 ,\,x_1x_4,\,  x_3x_4].$
Let $Z=\left(
         \begin{array}{cc}
           z_1 & z_2 \\
           z_3 & z_4 \\
         \end{array}
       \right) \in \Omega_{2,2}$. Then
\begin{eqnarray}
&&Z^\# = \left\{[x_1,\,x_2,\, x_1z_1 + x_2z_3,\, x_1z_2 + x_2z_4] \in D_{2,2} : [x_1,\,x_2]\in \mathbb{P}^1\right\}\nonumber \\
&&g([x_1,\,x_2,\, x_1z_1 + x_2z_3,\, x_1z_2 + x_2z_4]) = [A,\,B] \text{ where }\nonumber\\
&&A=(x_1^2,\, x_1x_2,\, x_2(x_1z_1 + x_2z_3)),\nonumber\\
&&B= ((x_1z_1 + x_2z_3)^2 ,\,x_1(x_1z_2 + x_2z_4),\,  (x_1z_1 + x_2z_3)(x_1z_2 + x_2z_4)).\nonumber
\end{eqnarray}
Denote $f(Z) = \left(
                 \begin{array}{ccc}
                   L_1& M_1& N_1 \\
                    L_2& M_2& N_2 \\
                    L_3& M_3& N_3 \\
                 \end{array},
               \right)$
then
\begin{eqnarray}
 x_1^2L_1+x_1x_2L_2 + x_2(x_1z_1 + x_2z_3)L_3 &=& (x_1z_1 + x_2z_3)^2 \nonumber\\
 x_1^2M_1+x_1x_2M_2 + x_2(x_1z_1 + x_2z_3)M_3 &=&x_1(x_1z_2 + x_2z_4) \nonumber\\
 x_1^2N_1+x_1x_2N_2 + x_2(x_1z_1 + x_2z_3)N_3&=&(x_1z_1 + x_2z_3)(x_1z_2 + x_2z_4) \nonumber
\end{eqnarray}
for all $[x_1,x_2]\in \mathbb{P}^1$. Hence we obtain \eqref{map me}.

Consider the case \eqref{standard} in Theorem \ref{classify ball}.
Suppose that for simplicity suppose that $g : D_{2,2}\rightarrow D_{3,3}$ is $g(x) = x_1$.
This method can be applied to general $r$ and homogeneous monomial linear map $g$.
The induce map $f :\Omega_{2,2}\rightarrow \Omega_{3,3}$ has the form
$$f\left(\left(
         \begin{array}{cc}
           z_1 & z_2 \\
           z_3 & z_4 \\
         \end{array}
       \right)\right)
       = \left(\begin{array}{ccc}
           z_1-L & z_2-M &1-N \\
           z_3 & z_4 & 0 \\
           L & M & N \\
         \end{array}
       \right)
       $$
for some holomorphic functions $L,M,N$ on $\Omega_{2,2}$.
Since $ I_{3,3} - f(Z)f(Z) > 0$, for $V = (v,1,0)$ and $Z$ in the Shilov boundary of $\Omega_{2,2}$,
\begin{eqnarray} \label{at shilov boundary}
0 &\leq& V( I_{3,3} - f(Z)f(Z) )V^* \nonumber\\
&=& 1 + v^2 - | v(z_1-L) + z_3|^2 -| v(z_2-M) + z_4|^2 - |v(1-N)|^2\nonumber\\
&=& -v^2(|L|^2 + |M|^2 + |N|^2) + (\text{ first order term in } v, \overline{v})
\end{eqnarray}
As $v\rightarrow \infty$, \eqref{at shilov boundary} tends to $-\infty$,
if one of $L,M,N$ are nonzero at $Z$.
This implies that $L,M,N$ should be zero on the Shilov boundary of $\Omega_{2,2}$ and
hence $L = M = N =0$ on $\Omega_{2,2}$.
However in this case $f$ is not a holomorphic map into $\Omega_{3,3}$.
Thus there is no proper holomorphic map induced from $g$ with nonzero $\phi$.

If $g(x) = [x_1,x_2,0,x_3,x_4,0]$, the induced map $f$ is given by $$f(Z) = \left(
         \begin{array}{cc}
           Z & 0 \\
           k(Z) & h(Z) \\
         \end{array}
       \right)$$
for some holomorphic functions $k_1,k_2,h$ on $\Omega_{2,2}$ where $k=(k_1,k_2)$.
Then by considering $f$ on the Shilov boundary as the same method above, $k$ should be zero.
Hence $f$ should be of the form \ref{standard map} in Theorem \ref{main theorem}.

\begin{remark}
Note that in generally, for one $g$, there could be several $f$. However, in case of $D_{2,2},\, D_{3,3}$ and $\Omega_{2,2},\,\Omega_{3,3}$, there is a unique $f$ for each $g$ since the number of equations and the number of unknowns are same.
\end{remark}

\subsection{More examples}
\begin{example}
If the difference of dimension gets bigger, then there are infinite number of proper holomorphic maps which are not rationally equivalent up to the automorphisms.
Consider the proper holomorphic maps from $D_{2,2}$ to $D_{4,4}$. As the same method, let $P_t(x) = (x_1+x_2-x_3-x_4)\, Q_P(x)$ where
$Q_{P_t}(x) = x_1+x_2+x_3+x_4-t(x_2+x_4)$ where $0 \leq t\leq 1$. Then
$$P_t(x) = x_1^2 + (2-t)x_1x_2 + (1-t)x_2^2 + t x_2x_3 - x_3^2 - (2-t)x_3x_4 - (1-t)x_4^2 - t x_1x_4$$ and the induced proper holomorphic maps are
$$g_t([z_1,z_2,z_3,z_4]) = [ z_1^2 ,\sqrt{2-t}z_1z_2, \sqrt{1-t}z_2^2, \sqrt{t} z_2z_3, z_3^2 ,\sqrt{2-t}z_3z_4,\sqrt{1-t}z_4^2,\sqrt{t} z_1z_4 ].$$
This $g_t$ satisfies the condition in Proposition \ref{indep vectors} and
 $g_t$ induces infinite number of proper holomorphic maps from $f_t : \Omega_{2,2}\rightarrow \Omega_{4,4}$ which is defined by
\begin{equation}\label{inequivalent example}
  \left(
         \begin{array}{cc}
           z_1 & z_2 \\
           z_3 & z_4 \\
         \end{array}
       \right) \mapsto \left(
         \begin{array}{cccc}
           z_1^2 & \sqrt{2-t}z_1z_2 & \sqrt{1-t} z_2^2& \sqrt{t}z_2 \\
           \sqrt{2-t}z_1z_3 & \frac{\sqrt{2-t}-t}{\sqrt{2-t}}z_1z_4 + z_2z_3& 2\sqrt{\frac{1-t}{2-t}}z_2z_4& \sqrt{\frac{t}{2-t}}z_4 \\
           \sqrt{1-t}z_3^2 & \frac{\sqrt{2-t}-t}{\sqrt{1-t}}z_3z_4& z_4^2& 0 \\
           \sqrt{t}z_3 & \sqrt{t}z_4&0&0\\
         \end{array}
       \right).
\end{equation}
\begin{remark}
\eqref{map me} and \eqref{map ng} are homotopic to each other by \eqref{inequivalent example}.
\end{remark}
\end{example}

\begin{example}
There are proper holomorphic map $f: \Omega_{2,2}\rightarrow \Omega_{4,4}$ which has degree 3 polynomial in components. Let $Q_P(x) = x_1^2 + x_1x_3 + x_3^2$. Then $P(x) = x_1^3 + x_1^2 x_2 + x_1x_2x_3 + x_2x_3^3 -x_3^2 -x_1^2 x_4 - x_1x_3x_4 - x_3^2x_4$ and hence $$g([x_1,\,x_2,\,x_3,\,x_4]) = [x_1^3,\, x_1^2 x_2,\, x_1x_2x_3,\, x_2x_3^3,\, x_3^2,\, x_1^2 x_4,\, x_1x_3x_4,\, x_3^2x_4].$$ The corresponding proper holomorphic map $f :\Omega_{2,2}\rightarrow \Omega_{4,4}$ is
$$  \left(
         \begin{array}{cc}
           z_1 & z_2 \\
           z_3 & z_4 \\
         \end{array}
       \right) \mapsto
        \left(
         \begin{array}{cccc}
           z_1^3 & z_2 & z_1z_2& z_1^2z_2 \\
           z_1^2z_3^2 & z_4 & z_2z_3 & z_1z_2z_3 - z_1^2z_4 + z_1^2z_3z_4 \\
           3z_1z_3 -2z_1z_3^2 & 0& 0 &z_2z_3 + 2z_1z_4-2z_1z_3 z_4 \\
           z_3^2 & 0&0&z_3 z_4 \\
         \end{array}
       \right).
$$
\end{example}

\begin{example}[Generalized Whitney map]
Consider
$$ P(z) = (x_1 +\dots + x_r -x_{r+1} -\dots - x_{r+s} )(x_1+x_{r+1}).$$
This polynomial induces the proper meromorphic map $g : D_{r,s} \rightarrow D_{2r-1,2s-1}$ defined by
\begin{eqnarray*}
g([z_1,\dots, z_r, w_1,\dots, w_s]) = [z_1^2,\, z_1z_2, \dots , z_1z_r,\, w_1z_2,\dots, w_1z_r, \\
w_1^2,\, w_1w_2, \dots, w_1w_s,\, z_1w_2, \dots z_1w_s].
\end{eqnarray*}
$g$ induces the proper holomorphic map $f^w : \Omega_{r,s} \rightarrow \Omega_{2r-1,2s-1}$ defined by
\begin{equation} \label{generalized whitney map}
 \left(
  \begin{array}{ccc}
    z_{11} & \dots & z_{1s} \\
    \vdots & \ddots & \vdots \\
     z_{r1} & \dots & z_{rs} \\
  \end{array}
\right) \mapsto \left(
  \begin{array}{ccccccc}
    z_{11}^2  & z_{11}z_{12}&\dots& z_{11}z_{1s}& z_{12}& \dots & z_{1s}   \\
     z_{21}z_{11} & z_{21}z_{12}& \dots& z_{21}z_{1r} &z_{22}& \dots& z_{2s}  \\
    \vdots & \vdots &\ddots & \vdots&\vdots& \ddots &\vdots \\
    z_{r1}z_{11} & z_{r1}z_{12} &\dots & z_{r1}z_{1s}& z_{r2}& \dots& z_{rs} \\
    z_{21} & z_{22} &\dots & z_{2s}& 0& \dots & 0 \\
     \vdots& \vdots & \ddots& \vdots & \vdots& \ddots & \vdots\\
     z_{r1}& z_{r2}&\dots &z_{rs}&0& \dots& 0\\
  \end{array}
\right)
\end{equation}
This is generalized proper holomorphic map of \eqref{map me} : if $r=s=2$, $f^w$ is same with \eqref{map me} in Theorem \ref{main theorem}.
\end{example}
\begin{example}
Consider the proper holomorphic maps from $D_{2,2}$ to $D_{3,4}$.
Let $P_t(x) =  (x_1+x_2-x_3-x_4)\, Q_P(x)$ where
$Q_{P_t}(x) = x_1+tx_3$ where $0 \leq t\leq 1$. Then proper rational map $g_t : D_{2,2}\rightarrow D_{3,4}$ is given by
$$
g_t([x_1,\,x_2,\,x_3,\,x_4]) = [ x_1^2,\, x_1x_2,\, \sqrt{t}x_2x_3,\, \sqrt{t}x_3^2, \, \sqrt{t}x_3x_4,\, \sqrt{1-t}x_1x_3, \, x_1x_4 ].
$$
The induced proper holomorphic maps $f_t :\Omega_{2,2} \rightarrow \Omega_{3,4}$ is given by
\begin{equation}\label{inequivalent example2}
 \left(
         \begin{array}{cc}
           z_1 & z_2 \\
           z_3 & z_4 \\
         \end{array}
       \right) \mapsto
\left(
         \begin{array}{cccc}
           \sqrt{t}z_1^2 & \sqrt{t}z_1z_2 & \sqrt{1-t} z_1& z_2 \\
           \sqrt{t}z_1z_3 & \sqrt{t}z_2z_3& \sqrt{1-t}z_3& z_4 \\
           z_3& z_4& 0 & 0 \\
         \end{array}
       \right).
\end{equation}
Furthermore we can generalize proper holomorphic map \eqref{inequivalent example2} to
$F_t :\Omega_{r,s}\rightarrow \Omega_{2r-1,2s}$ given by for $Z = (z_{ij})_{1\leq i\leq r,\, 1\leq j\leq s}$,
\begin{equation} \label{generalized whitney map}
 Z
\mapsto\left(
  \begin{array}{cccccccc}
    \sqrt{t}z_{11}^2  & \sqrt{t}z_{11}z_{12}&\dots& \sqrt{t}z_{11}z_{1s}&\sqrt{1-t}z_{11}& z_{12}& \dots & z_{1s}   \\
    \sqrt{t}z_{11} z_{21} & \sqrt{t}z_{21}z_{12}& \dots& \sqrt{t}z_{21}z_{1r} &\sqrt{1-t}z_{21}&z_{22}& \dots& z_{2s}  \\
    \vdots & \vdots &\ddots & \vdots&\vdots& \ddots &\vdots \\
    \sqrt{t}z_{11}z_{r1} & \sqrt{t}z_{r1}z_{12} &\dots & \sqrt{t}z_{r1}z_{1s}&\sqrt{1-t}z_{r1}& z_{r2}& \dots& z_{rs} \\
    z_{21} & z_{22} &\dots & z_{2s}& 0&0& \dots & 0 \\
     \vdots& \vdots & \ddots& \vdots &\vdots& \vdots& \ddots & \vdots\\
     z_{r1}& z_{r2}&\dots &z_{rs}&0&0& \dots& 0\\
  \end{array}
\right)
\end{equation}
\end{example}
\begin{example}
Consider
$$
P(x) = (x_1+\dots + x_r -y_1-\dots -y_s)(x_1+\dots + x_r + y_1 + \dots + y_s)
$$
and the induced rational map $g : D_{r,s} \rightarrow D_{r',s'} $ where $ r' = \frac{1}{2}r(r+1)$, $ s'=\frac{1}{2}s(s+1)$ defined by
\begin{eqnarray}
 g([x_1, \dots, x_r, y_{1},\dots, y_{s}]) &=& [x_1^2,\dots, x_r^2,\, \sqrt{2}x_1x_2, \dots, \sqrt{2}x_ix_j,\dots, \sqrt{2}x_{r-1}x_r,\nonumber \\
 &&y_1^2,\dots, y_s^2,\, \sqrt{2}y_1y_2, \dots, \sqrt{2}y_ky_l,\dots, \sqrt{2}y_{s-1}y_s]\nonumber
\end{eqnarray}
where $i,j,k$ and $l$ trace over $1\leq i<j \leq r$ and $1\leq k<l \leq s$. Then the induced proper holomorphic map $f : \Omega_{r,s} \rightarrow \Omega_{r',s'}$ is
\begin{eqnarray*}
f\left(
\left(
  \begin{array}{ccc}
    z_{11} & \dots & z_{1s} \\
    \vdots & \ddots & \vdots \\
    z_{r1} & \dots & z_{rs} \\
  \end{array}
\right)
\right)
= (M,N)\,\,
\text{ where }\,\,
M = \left(
  \begin{array}{ccc}
    z_{11}^2 & \dots & z_{1s}^2  \\
    \vdots &  & \vdots  \\
    z_{r1}^2 & \dots & z_{rs}^2 \\
    \sqrt{2}z_{11}z_{21} & \dots & \sqrt{2}z_{1s}z_{2s} \\
    \vdots &  & \vdots  \\
    \sqrt{2}z_{i1}z_{j1} & \dots & \sqrt{2}z_{is}z_{js} \\
    \vdots &  & \vdots \\
 \sqrt{2}z_{r-1 1}z_{r1} & \dots & \sqrt{2}z_{r-1 s}z_{rs}  \\
  \end{array}
\right)
\end{eqnarray*}
$$
N= \left(
  \begin{array}{ccccc}
  \sqrt{2}z_{11} z_{12} & \dots & \sqrt{2}z_{1k}z_{1l}  & \dots & \sqrt{2}z_{1 s-1}z_{1 s} \\
     \vdots &  & \vdots &  & \vdots \\
      \sqrt{2}z_{r1} z_{r2} & \dots & \sqrt{2}z_{r k}z_{r l}  & \dots & \sqrt{2}z_{r s-1}z_{rs} \\
    z_{11}z_{22} + z_{12}z_{21} & \dots & z_{1k}z_{2l}+z_{2k}z_{1l} & \dots & z_{1{s-1}}z_{2 s} +z_{2{s-1}}z_{1s} \\
     \vdots &  & \vdots &  & \vdots \\
    z_{i1}z_{j2} + z_{j1}z_{i2} & \dots & z_{ik}z_{jl}+z_{jk}z_{il} & \dots & z_{i{s-1}}z_{js}+z_{j{s-1}}z_{is} \\
     \vdots &  & \vdots &  & \vdots \\
  z_{{r-1} 1} z_{r2} + z_{r1}z_{r-1 2} & \dots & z_{r-1 k}z_{rl}+z_{rk} z_{r-1 l} & \dots & z_{r-1 s-1}z_{rs}+z_{r s-1}z_{r-1 s} \\
  \end{array}
\right).
$$
Here $i,j,k,l$ trace over $1\leq i < j \leq r$ and $1\leq k < l \leq r$.
\end{example}

\end{document}